\documentclass[12pt]{article}
\usepackage{amsmath,latexsym,amsthm,amsfonts}
\usepackage{amssymb,longtable,amscd}
\newtheorem{theorem}{Theorem}

\newtheorem{lemma}{Lemma}

\usepackage[pdftex]{hyperref}
\usepackage[margin=1.25in]{geometry}
\usepackage[font=small, margin=.75cm, justification=centering]{caption}
\usepackage{multicol,tikz}

\usepackage{bigdelim,multirow}

\newcommand{\ve}{{\sc VertexExtend} }
\newcommand{\veb}{{\sc VertexExtend}}
\newcommand{\gl}{{\sc Glue} }
\newcommand{\glb}{{\sc Glue}}

\newcommand{\abs}[1]{\lvert #1 \rvert} %%for absolute value sign, better than ||
\newcommand{\ex}{\text{ex}}

\begin{document}

\title{Computation of the Ramsey Numbers\\ $R(C_4,K_9)$ and $R(C_4,K_{10})$}
\author{
Ivan Livinsky\\
\small Department of Mathematics \\[-0.8ex]
\small University of Toronto\\ [-0.8ex]
\small Toronto, ON M5S 2E4\\ [-0.8ex]
\small \texttt{ivan.livinskyi@mail.utoronto.ca}
\\ [.6em]
%and\\
\\ [-.6em]
Alexander Lange,  Stanis\l aw Radziszowski \\
\small Department of Computer Science \\[-0.8ex]
\small Rochester Institute of Technology \\[-0.8ex]
\small Rochester, NY 14623 \\[-0.8ex]
\small \texttt{\{arl9577,spr\}@cs.rit.edu}\\
}
\date{}
\maketitle

\begin{abstract}
The Ramsey number $R(C_4,K_m)$ is the smallest $n$ such that any graph on $n$
vertices contains a cycle of length four or an independent set of order
$m$.  With the help of computer algorithms we obtain the exact values of the
Ramsey numbers $R(C_4,K_9)=30$ and $R(C_4,K_{10})=36$. New bounds for
the next two open cases are also presented.
\end{abstract}

\section{Introduction}
Let $G$ and $H$ be simple graphs. An $n$-vertex graph $F$ is a $(G,H;n)$-graph
if it contains no subgraph isomorphic to $G$ and $\overline{F}$ contains no
subgraph isomorphic to $H$. Define $\mathcal{R}(G,H;n)$ to be the set of all
such graphs.  The Ramsey number $R(G,H)$ is the smallest $n$ such that for every
two-coloring of the edges of $K_n$, a monochromatic copy of $G$ or $H$ exists in
the first or second color, respectively. Clearly, if a $(G,H;n)$-graph exists,
then $R(G,H) > n$. It is known that Ramsey numbers exist \cite{Ramsey1930} for
all $G$ and $H$. The values and bounds for various types of such numbers are
collected and regularly updated by the third author \cite{Radziszowski2011b}.

The cycle-complete Ramsey numbers $R(C_n,K_m)$ have received much attention,
both theoretically and computationally. For fixed $n=3$, the problem becomes
that of
$R(3,k)$, which has been widely studied (see for example \cite{Spencer2011}),
including exact determination of its asymptotics \cite{Kim1995}. Since 1976, it
has been conjectured that $R(C_n,K_m)=(n-1)(m-1)+1$ for all $n \geq m \geq 3$,
except $n=m=3$ \cite{Faudree1978,efrs1978}.  Note that the lower bound is easy:
$(m-1)$ vertex-disjoint copies of $K_{n-1}$ provides a witness for $R(C_n,K_m) >
(n-1)(m-1)$. For over 35 years, much work has been done to verify the upper
bound, with $m=8$ being the current smallest open case.
\vspace*{1em}

This work involves fixed $n=4$, that is, the case of avoiding the quadrilateral
$C_4$ in the first color. The currently best known asymptotic bounds for
$R(C_4,K_m)$ are as follows.

\begin{theorem}[\cite{Spencer1977,Caro2000b}] There exist positive constants $c_1$ and $c_2$ such that
  \begin{equation*} \label{eq:c4bounds}
    c_1\left(\frac{m}{\log{m}}\right)^{\frac{3}{2}} \leq R(C_4,K_m) \leq c_2
    \left( \frac{m}{\log{m}} \right)^2.
  \end{equation*}
\end{theorem}

\noindent The lower bound was obtained by Spencer in 1977 \cite{Spencer1977}
using the probabilistic method. The upper bound was published by
Caro, Li, Rousseau, and Zhang in 2000 \cite{Caro2000b}, who in turn gave credit
to an unpublished work by Szemer\'{e}di. The main challenge
is determining whether $R(C_4,K_n) < n^{2-\epsilon}$ for some $\epsilon > 0$, a
question posed by Erd\H{o}s in 1981 \cite{Erdos1981}.

Prior to this work, the exact values for $R(C_4,K_m)$ were known for $3 \leq m
\leq 8$. Here, we present a computational proof that $R(C_4,K_9)=30$ and
$R(C_4,K_{10})=36$. The known values and bounds, including our new results, are
gathered in Table \ref{tab:values}.

%\begin{table}
% \centering
%  \begin{tabular}{ | c | c c c c c c c c c c|}\hline
%    $m$ & 3 & 4 & 5 & 6 & 7 & 8 & 9 & 10 & 11 & 12\\ \hline
%    $R(C_4,K_m)$ & 7 & 10 & 14 & 18 & 22 & 26 & {\bf 30} &
%    {\bf 36--37} & {\bf 39--45} & {\bf 42--54}\\ \hline  
%  \end{tabular}
%  \caption{Known values and bounds for $R(C_4,K_m)$ \cite{Radziszowski2011b}}
%  \label{tab:values}
%\end{table}

\begin{table}[h]
  \centering
  \begin{tabular}{ r c c c } \hline
    $m$ & $R(C_4,K_m)$ & Year & References \\ \hline
    3 & 7  & 1971 & \cite{Chartrand1971} \\
    4 & 10 & 1972 & \cite{Chvatal1972}\\
    5 & 14 & 1977 & \cite{Clancy1977} \\
    6 & 18 & 1987/1977 & \cite{Exoo1987}/\cite{Rousseau1977}\\
    7 & 22 & 2002/1997 & \cite{Radziszowski2002}/\cite{Jayawardene1998}\\
    8 & 26 & 2002 & \cite{Radziszowski2002} \\ \hline
    9 & 30 & \multicolumn{2}{c}{\multirow{4}{*}{this work}} \\
    10 & 36 &    \\
    11 & 39--44 & \\
    12 & 42--53 & \\ \hline
  \end{tabular}
  \caption{Known values and bounds for $R(C_4,K_m)$. \\Double references
    correspond to lower and upper bounds.}
    \label{tab:values}
\end{table}

The value of $R(C_4,K_6)$ and bounds $21\leq R(C_4,K_7)\leq 22$ were presented
by Jayawardene and Rousseau in \cite{Jayawardene1998,Jayawardene2000}. The numbers $R(C_4,K_7)$,
$R(C_4,K_8)$ and the bounds $30\leq R(C_4,K_9)\leq 33$, $34\leq R(C_4,K_{10})\leq
40$ were given by Radziszowski and Tse in \cite{Radziszowski2002}. Further upper bound improvements
to $32$ and $39$ for $R(C_4,K_9)$ and $R(C_4,K_{10})$, respectively, were
presented in \cite{Xu2009b}. 
\vspace*{1em}

For graph
$G$, $V(G)$ is the vertex set; $E(G)$ is the edge set; $N_G(v)$ is the
neighborhood of $v \in V(G)$; $\deg_G(v)$ is $\abs{N_G(v)}$; $\delta(G)$ is the
minimum degree; and $\alpha(G)$ is the independence number.

\section{Algorithms and Computations}
\subsection{Higher Level}
The computations and algorithms used in this work are similar to those described
in \cite{Radziszowski2002}. Comparable methods have been used to find other Ramsey
numbers, such as in \cite{McKay1995,Goedgebeur2013}.

The main idea behind the computations is to enumerate the sets
$\mathcal{R}(C_4,K_m)$. If $\mathcal{R}(C_4,K_m;n) \neq \emptyset$, then
$R(C_4,K_m) > n$, and if $\mathcal{R}(C_4,K_m;n+1) = \emptyset$, then
$R(C_4,K_m) \leq n + 1$. The latter is usually accomplished by extending
$\mathcal{R}(C_4,K_m;t)$ to graphs in sets with higher $m$ and/or $t$.
Two methods used to achieve this are described in the next section.

Some special properties of $C_4$-free graphs proved useful during our
computations. One such property involves an extremal Tur\'{a}n-type problem
involving $C_4$-free graphs. Let $\ex(n,C_4)$ be the maximum number of edges of
an $n$-vertex $C_4$-free graph. These numbers have been studied extensively both
theoretically and computationally (cf. \cite{ExtremalGraphTheory}). The values
of $\ex(n,C_4)$ for $1 \leq n \leq 32$ are known
\cite{Clapham1989,Yuansheng1992,Shao2009} and they are displayed in Table
\ref{tab:ex}.

\begin{lemma}[\cite{Chvatal1972,ExtremalGraphTheory}]\label{lem:bounds}
  If a $C_4$-free graph has $n$ vertices, $e$ edges, and minimum degree $\delta$, then
   $\delta^2 - \delta + 1 \leq n$ and 
  $e < \frac{1}{4}\: n\: (1+\sqrt{4n-3})$.
\end{lemma}

\begin{table}
\centering
\begin{tabular}{ r | r r r r r r r r r r  }
  \multicolumn{1}{c|}{$n$} &  3 & 4 & 5 & 6 & 7 & 8 & 9 & 10 & 11 & 12\\ \hline
  \multicolumn{1}{c|}{$\ex(n,C_4)$} &  3 & 4 & 6 & 7 & 9 & 11 & 13 &
  16\ & 18 & 21\vspace*{1em}\\
  \multicolumn{1}{c|}{$n$} & 13 & 14 & 15 & 16 & 17 & 18 & 19 & 20 & 21 & 22\\ \hline
  \multicolumn{1}{c|}{$\ex(n,C_4)$} & 24 & 27 & 30 & 33 & 36 & 39 & 42 & 46 &
  50 & 52\vspace*{1em}\\
   \multicolumn{1}{c|}{$n$} & 23 & 24 & 25 & 26 & 27 & 28 & 29 & 30 & 31 & 32\\ \hline
  \multicolumn{1}{c|}{$\ex(n,C_4)$} & 56 & 59 & 63 & 67 & 71 & 76 & 80 & 85 & 90 & 92\\
\end{tabular}
\vspace*{1em}
\caption{Known values for $\ex(n,C_4)$ \cite{Clapham1989,Yuansheng1992,Shao2009}.}
\label{tab:ex}
\end{table}

\subsection{Methods}\label{sec:methods}
Our enumeration of various classes of $(C_4,K_m)$-graphs uses two computational
methods, \ve and \glb, described below.  

\vspace*{1em}

\noindent \ve
\vspace*{.5em}

This algorithm  extends a $(C_4,K_m;n)$-graph
$G$ to all possible $(C_4,K_m;n+1)$-graphs $G'$ containing $G$ by attaching a
new vertex $v$ to all feasible neighborhoods in $G$. By feasible, we mean that
the additional edges do not create a $C_4$ while also preserving $\alpha(G') <
m$. If complexity of computations is ignored, then full enumeration of
$\mathcal{R}(C_4,K_m;n+1)$ can clearly be obtained from $\mathcal{R}(C_4,K_m;n)$ with
this method.
\vspace*{1em}

\noindent{\sc Glue}
\vspace*{.5em}

The second method, called the {\sc Glue} algorithm, constructs
$\mathcal{R}(C_4,K_m;n+\delta+1)$ from $\mathcal{R}(C_4,K_{m-1};n)$, where
$\delta$ is the minimum degree of the new graphs. For a
$(C_4,K_{m};n+\delta+1)$-graph $G$, let $v \in V(G)$ be such that
$\deg_G(v)=\delta(G)$, and let $X$ be the subgraph induced by $N_G(v)$; $X$ must
be a $(P_3,K_{m};\delta)$-graph. Note that such a graph must be of the form
$sK_2 \cup tK_1$, where $2s+t=\delta$ and $s+t<m$. Let $Y$ be the induced subgraph of $V(G)
\setminus (X \cup \{v\})$; $Y$ must be a $(C_4,K_{m-1};n)$-graph. If we know
$\mathcal{R}(C_4,K_{m-1};n)$, we can find all graphs in
$\mathcal{R}(C_4,K_{m};n+\delta+1)$ by considering how each vertex $x \in X$ can
be connected to the vertices of $Y$. We call each neighborhood $N(x) \cap V(Y)$
the \emph{cone} of $x$, denoted $c(x)$. We say that the cone $c(x)$ is
\emph{feasible} if:
\begin{enumerate}
  \item $c(x)$ does not contain two endpoints of any $P_3$ in $Y$.
  \item For distinct $x_1,x_2 \in V(X)$, $c(x_1) \cap c(x_2) = \emptyset$.
  \item For each edge $\{x_1,x_2\} \in E(X)$, there is no $y_1 \in
    c(x_1)$ and $y_2 \in c(x_2)$ such that $\{y_1,y_2\} \in E(Y)$.
  \item For each subgraph induced by $X' \subseteq X$ and 
     $Y'$ induced by $V(Y) \setminus \bigcup_{x \in X'} c(x)$, 
    $\alpha(X')+\alpha(Y') < m$.
\end{enumerate}

\noindent Conditions 1, 2, and 3 prevent $C_4$'s, while condition 4 prevents 
independent sets of order $m$. Figure \ref{fig:c4_glue} presents the main idea
of {\sc Glue}.

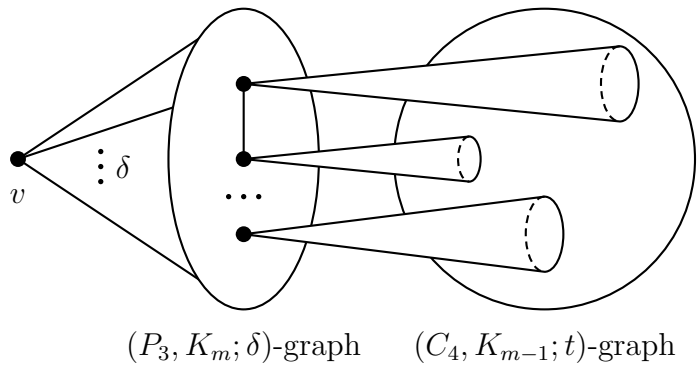
\begin{figure}
\centering
\vspace*{1em}
%\usetikzlibrary{shapes,snakes}
{\begin{tikzpicture}[style=thick]
  \fill[color=black] (1,1) circle (3pt);
  \node at (1,.5) {$v$};
  \draw (1,1)--(4,3);
  \draw (1,1)--(4,-1);
  \draw (1,1)--(4,2);
  %\draw (4,2)--(9,2.5);
  %\draw (4,2)--(9,1.5); 
  %\draw (9,2.5)--(9,1.5);
  %\draw (4,0)--(8,.5);
  %\draw (4,0)--(8,-.5);
  %\draw (8,.5)--(8,-.5);
  %\draw (4,1)--(7,1.3);
  %\draw (4,1)--(7,.7);
  %\draw (7,1.3)--(7,.7);
  \draw[fill=white] (4,1) ellipse (1cm and 2cm);
  \draw (8,1) circle (2cm);
  \fill[color=black] (2.1,1.1) circle (1pt);
  \fill[color=black] (2.1,.9) circle (1pt);    
  \fill[color=black] (2.1,.7) circle (1pt);
  
  \fill[color=white] (9,2.5) arc (90:-90:0.25cm and 0.5cm) -- (4,2) -- cycle;
  \draw (9,2.5) arc (90:-90:0.25cm and 0.5cm) -- (4,2) -- cycle;
  \draw[densely dashed] (9,2.5) arc (90:270:0.25cm and 0.5cm);
  \fill[color=black] (4,2) circle (3pt);
  
  \fill[color=white] (8,.5) arc (90:-90:0.25cm and 0.5cm) -- (4,0) -- cycle;
  \draw (8,.5) arc (90:-90:0.25cm and 0.5cm) -- (4,0) -- cycle;
  \draw[densely dashed] (8,.5) arc (90:270:0.25cm and 0.5cm);
  \fill[color=black] (4,0) circle (3pt);
  
  \fill[color=white] (7,1.3) arc (90:-90:0.15cm and 0.3cm) -- (4,1) -- cycle;
  \draw (7,1.3) arc (90:-90:0.15cm and 0.3cm) -- (4,1) -- cycle;
  \draw[densely dashed] (7,1.3) arc (90:270:0.15cm and 0.3cm);
  \fill[color=black] (4,1) circle (3pt);
  
  \fill[color=black] (4,0.5) circle (1pt);
  \fill[color=black] (4.2,0.5) circle (1pt);    
  \fill[color=black] (3.8,0.5) circle (1pt);

  \draw (4,2)--(4,1);
  
  \node at (2.4,.9) {$\delta$};
  \node at (4,-1.5) {$(P_3,K_m; \delta )$-graph};
  \node at (8,-1.5) {$(C_4,K_{m-1};t)$-graph};
\end{tikzpicture}}
\caption{Gluing to a $(C_4,K_m;\delta+t+1)$-graph.}
\label{fig:c4_glue}
\end{figure}

\subsection{Implementation and Optimization}
Two separate implementations of \ve and \gl were used in order to corroborate
the correctness of the results. In all cases where both implementations were
used, the results agreed. We list the details of this agreement in the Appendix.

The rules for {gluing} $(C_4,K_m)$-graphs described in Section \ref{sec:methods}
allowed for a much needed speedup in computations. In most cases, it was
beneficial to preprocess the $Y$ graphs before gluing, storing information about
the feasibility of the cones. For example, all subsets of vertices containing
endpoints of a $P_3$ were removed from the list of feasible cones. Speed was
greatly increased by precomputing the independence number $\alpha(Y')$ of each subgraph,
which was critical for efficient testing of condition 4. This
proved to be a bottleneck of the computations, and multiple strategies and
implementations were tested. The most efficient algorithm implemented was based on {\it
  Algorithm 1: Precomputing independence number}, described in
\cite{Goedgebeur2013}. All data was stored in arrays of size $2^n$, where the
integer index of the array represented the bit-set of the vertices of the
subgraph.

Two isomorphism testing tools were used in our implementations. The first 
implemented an algorithm described by William Kocay \cite{Kocay1996}. The other
 made use of the well-known software {\tt nauty} by Brendan McKay
\cite{nauty}.

\section{Results}
First, we obtained a full enumeration of $\mathcal{R}(C_4,K_7)$. This was
significant, as the same enumeration was computationally infeasible when these
methods were attempted in 2002
\cite{Radziszowski2002}. $\mathcal{R}(C_4,K_7)$ was first obtained using
\veb. The same results were obtained when gluing from $\mathcal{R}(C_4,K_6)$. For more
information on these and similar consistency checks, see the Appendix.
The statistics of $\mathcal{R}(C_4,K_7)$ by vertex and edge counts are displayed
in Tables \ref{tab:7a} and \ref{tab:7b}. The cases of counts found in
\cite{Radziszowski2002} agree with ours.

\begin{table}
  {\tabcolsep=0.30em \small
  \centering
  \begin{tabular}{r c |rrrrrrrrr}
     & $n$ & $7$ & $8$ & $9$ & $10$ & $11$ & $12$ & $13$ & $14$ & $15$\\
    $e$ &&&&&&&&&&\\ \hline
    $1$ && 1& & & &  & & & & \\
    $2$ && 2& 1& & & & & & & \\
    $3$ && 5& 4& 1& & & & & & \\
    $4$ && 9& 9& 4&  1 & & & & & \\
    $5$ && 18& 20&14&4 & 1 & & & & \\
    $6$ && 29& 42&40&16& 3 &1 & & & \\
    $7$ && 30& 71&91&57 &13 &2& & & \\
    $8$ && 17& 88&178&172&56&9&1& & \\
    $9$ && 5& 72&274&422&221&41&4& & \\
    $10$ && & 31&289&805&737&183&19& 1& \\
    $11$ && & 5&197&1135&1947&779&94&5& \\
    $12$ && &  &74&1097&3861&2912&469&28&1 \\
    $13$ && &  &10&670&5405&8660&2221&151&5 \\
    $14$ && &  &  &222&5046&18943&9455&826&29 \\
    $15$ && &  &  &34&2965&28496&32805&4367&163 \\
    $16$ && &  &  &2&971&27902&84467&21211&920 \\
    $17$ && &  &  & &146&16897&148686&87187&5218  \\
    $18$ && &  &  & &11&5831&168441& 277608&27740\\
    $19$ && &  &  & &  &1013&116266& 622072&130043 \\
    $20$ && &  &  & &  &82&45788& 904916&507036\\
    $21$ && &  &  & &  &3&9434&801944&1513611 \\
    $22$ && &  &  & &  & &916&406222&3119854\\
    $23$ && &  &  & &  & &39&108749&4033237\\
    $24$ && &  &  & &  & &2&14039&3021620\\
    $25$ && &  &  & &  & & &818&1215627\\
    $26$ && &  &  & &  & & &24&241075\\
    $27$ && &  &  & &  & & &1&21639\\
    $28$ && &  &  & &  & & & &851\\
    $29$ && &  &  & &  & & & &22\\
    $30$ && &  &  & &  & & & &2\\
    \hline
    \multicolumn{2}{c|}{Total} & 116& 343&1172 &4637 &21383 &111754 &619107 &3250169 &13838693 \\% \hline
  \end{tabular}
  \caption{Statistics for $\mathcal{R}(C_4,K_7;n)$, $7 \leq n \leq 15$.\\ Note
    that for $n < 7$ the counts would be for all $C_4$-free graphs.}
  \label{tab:7a}
  }
\end{table}

\begin{table}
\centering
{
\renewcommand{\arraystretch}{1}
\small
\begin{tabular}{rc|rrrrr @{$\quad\;\;\;$}r}
  & $n$ & $16$ & $17$ & $18$ & $19$ & $20$ & $21$ \\
  $e$ & & & & & &\\ \hline
  $14$ & &1 & & & &\\
  $15$ & &5 & & & &\\
  $16$ && 23 & 1 &  & & &\\
  $17$ && 116& 3 & &  & &\\
  $18$ && 644& 11& 1& & &\\
  $19$ && 3602& 51& 1& & &\\
  $20$ && 19588& 251& 3&  & &\\
  $21$ && 97521& 1311& 12& & & \\
  $22$ && 423964& 6805& 45& &  &\\
  $23$ && 1543985& 33476& 198&  & &\\
  $24$ && 4434855& 149441&908 & & & \\
  $25$ && 9068568& 585687&4045 & &  &\\
  $26$ && 11612126& 1964782&16971 &  &  &\\
  $27$ && 8299450& 5448131&64462 &  & &\\
  $28$ && 3016205& 11583843&219831 &  & &\\
  $29$ && 511367&  16465694&672324 &1  & &\\
  $30$ && 37318&   13277929&1813931 &18  & &\\
  $31$ && 1167&    5287770&4096321 &233  & &\\
  $32$ && 26&      938464&6953952 &2399 &  & \\
  $33$ && 2&       68369&7533349 &17474 &  & \\
  $34$ && &2018&4275886 &83786 &  &\\
  $35$ && &35&1064229 &261093  & &\\
  $36$ && &1&102512 &520551 & & \\
  $37$ && & &3512 & 605219 &1& \\
  $38$ && & &53 &328849 &12 & \\
  $39$ && & &1 &64919 &126 & \\
  $40$ && & &  &4132& 999 & \\
  $41$ && & &  &107 & 3611& \\
  $42$ && & &  &4 &3762 & \\
  $43$ && & &  &  &897&  \\
  $44$ && & &  &  &53 & \\
  $45$ && & &  &  &2 &1 \\
  $46$ && & &  &  &  &2 \\
  \hline
  \multicolumn{2}{c|}{Total} &39070533& 55814073& 26822547& 1888785& 9463&3\\
% \hline
\end{tabular}
}
\caption{Statistics for $\mathcal{R}(C_4,K_7;n)$, $16 \leq n \leq 21$.}
\label{tab:7b}
\end{table}

\vspace*{1em} 

Once $\mathcal{R}(C_4,K_7)$ was obtained, we were able to construct
$\mathcal{R}(C_4,K_8;n)$ for $n$ equal to $23$, $24$, and $25$. The gluing of
$\mathcal{R}(C_4,K_8;23)$ turned out to be the most computationally expensive,
as there are 353015495 such graphs, but this was needed in order to extend them
further to
$\mathcal{R}(C_4,K_9;29)$. The counts for $\mathcal{R}(C_4,K_8;23)$ are
displayed by size and minimum degree in Table \ref{tab:8a}. Statistics for
$\mathcal{R}(C_4,K_8;24)$ and $\mathcal{R}(C_4,K_8;25)$ are gathered in Table
\ref{tab:8b}.
Our computations found that no $(C_4,K_8)$-graph exists with minimum degree
$5$. %A non-computational proof follows.

\begin{table}
\centering
{
\begin{tabular}{rc | r r r r || r }
& $\delta$&	1&	2	&	3&		4&		Total\\
$e$ & & & & & &\\ \hline
40&&	&	1	&		&	&	 	1\\
41&&	&	13&			&	&	 	13\\
42&&	&	201	&		&		 &	201\\
43&&	&	3055	&	108	&		 &	3163\\
44&&	&	36884	&	8517	&		 &	45401\\
45&&	&	302179	&	260678	&		 &	562857\\
46&&	1&	1449548	&	3502385	&	83	 &	4952017\\
47&&	6&	3662039	&	23059729&	35368	 &	26757142\\
48&&	29&	4576213	&	75076644&	1563123	 &	81216009\\
49&&	53&	2716695	&	110589375&	11348103&	124654226\\
50&&	27&	744258	&	66302337&	19535975&	86582597\\
51&&	3&	95358	&	15327155&	9727032	 &	25149548\\
52&&	&	5827	&	1352590	&	1588719	 &	2947136\\
53&&	&	164	&	47152	&	94684	 &	142000\\
54&&	&	6	&	732	&	2404	 &	3142\\
55&&	&		&	4	&	37	 &	41\\
56&&	&		&		&	1	 &	1\\ \hline
\multicolumn{2}{c|}{Total}&	119&	13592441&	295527406&	43895529&	353015495
\end{tabular}}
\caption{Size vs minimum degree of graphs in $\mathcal{R}(C_4,K_8;23)$.\\ All such
graphs with $\delta = 4$ were used with \gl  to find $(C_4,K_9;29)$-graphs.}
\label{tab:8a}
\end{table}

\begin{table}
\centering
{%\renewcommand{\arraystretch}{0}
\begin{tabular}{rc|rr}
  & $n$ & $24$ & $25$  \\
  $e$ && & \\ \hline
  $48$ && 1& \\
  $49$ && 6& \\
  $50$ && 48& \\
  $51$ && 394& \\
  $52$ && 3133&  \\
  $53$ && 21116&  \\
  $54$ && 60646&\\
  $55$ && 57944& \\
  $56$ && 18863&  \\
  $57$ && 2102&  \\
  $58$ && 96& 2 \\
  $59$ && 4& 10 \\
  $60$ && & 15 \\
  $61$ && & 9 \\
  \hline
  \multicolumn{2}{c|}{Total} & 164353& 36\\

\end{tabular}}
\caption{Statistics for $\mathcal{R}(C_4,K_8;n)$, $n=24,25$.\\These graphs
  were used to find $(C_4,K_9;m)$-graphs for $m \geq 29$.}
\label{tab:8b}
\end{table}

\begin{table}
\centering
{
\begin{tabular}{r c | r r r || r }
 &$\delta$&	3&	4&	5&	 Total\\
$e$ & & & & &\\ \hline
70&&	1&	1&	&	  2\\
71&&	8&	5&	&	  13\\
72&&	12&	11&	&	  23\\
73&&	18&	33&	1&	  52\\
74&&	10&	64&	7&	  81\\
75&&	&	49&	9&	  58\\
76&&	&	19&	7&	  26\\
77&&	&	6&	4&	  10\\
78&&	&	&	2&	  2\\ \hline
 \multicolumn{2}{c|}{Total}&	49&	188&	30&       267
\end{tabular}}
\caption{Size vs minimum degree of graphs in $\mathcal{R}(C_4,K_9;29)$.\\ These
graphs were used to show that no $(C_4,K_{10};36)$-graph exists.}
\label{tab:9}
\end{table}

\subsection{$R(C_4,K_9)$}

\noindent We constructed the sets $\mathcal{R}(C_4,K_9;29)$ and
$\mathcal{R}(C_4,K_9;30)$ with the \gl algorithm. Since $R(C_4,K_8)=26$, any
$(C_4,K_9;29)$-graph has minimum degree $3$, $4$, or $5$ and can be
obtained from $\mathcal{R}(C_4,K_8;n)$ for $n=25,24,23$ by \glb. Note that the
minimum degree of a $(C_4,K_8;23)$-graph must be $4$ in order to glue to a graph
of minimum degree $5$. This restriction improved the speed of computation, as there is a
large number of $(C_4,K_8;23)$-graphs to consider. Statistics for
$\mathcal{R}(C_4,K_9;29)$ are found in Table \ref{tab:9}.

Similarly, any $(C_4,K_9;30)$-graph has minimum degree $4$ or $5$, and
can be obtained from $\mathcal{R}(C_4,K_8;25)$ or
$\mathcal{R}(C_4,K_8;24)$, respectively, via \gl. No
$(C_4,K_9;30)$-graphs were found, resulting in the following theorem. 

\begin{theorem} \label{thm:k9}
$R(C_4,K_9)=30$.
\end{theorem}

\subsection{$R(C_4,K_{10})$}

\begin{theorem}\label{thm:k10}
$R(C_4,K_{10}) = 36$.
\end{theorem}

\begin{proof}
We have found two $6$-regular $(C_4,K_{10};35)$-graphs $H_1$ and $H_2$,
establishing the lower bound. The
orbits of $H_1$ are depicted in Figure \ref{fig:liv} and its adjacency matrix is
presented in Figure \ref{fig:liv_adj}.

In order to prove $R(C_4,K_{10}) \leq 36$, it is necessary to show that no
$(C_4,K_{10};36)$-graph exists. As $R(C_4,K_9)=30$, from Lemma \ref{lem:bounds},
a $(C_4,K_{10};36)$-graph has minimum degree at most $6$ and can be obtained
from gluing a $(C_4,K_{9};29)$-graph. Gluing all of $\mathcal{R}(C_4,K_{9};29)$
resulted in finding no such graphs.
\end{proof}

The automorphism group $\mathrm{Aut}(H_1)$ has order $24$ and its action on
$V(H_1)$ has four orbits of $24$, $6$, $4$, and $1$ vertices, respectively. The
automorphism group $\mathrm{Aut}(H_2)$ has order $40$ and its action on $V(H_2)$
has three orbits of $20$, $10$, and $5$ vertices. Both graphs $H_1$ and $H_2$
have 105 edges and 35 triangles. Each vertex is on three triangles, that is,
each neighborhood is the union of three $K_2$ graphs. Both graphs are also
bicritical: removing any edge produces an independent set of order $10$, and
adding any edge produces a $C_4$.

Interestingly, no $(C_4,K_{10};n)$-graphs for $n=34,35$
were obtained  by gluing from $\mathcal{R}(C_4,K_{9};29)$.

\begin{figure}[h]
\centering
\scalebox{.8}{\begin{tikzpicture}

%\fill(-2.25,2) circle (4pt);
%\fill(-2.25,1.25) circle (4pt);
%\fill(-3,1.25) circle (4pt); 

\newcommand*{\h}{3.2}
\newcommand*{\triscale}{.5}
\newcommand*{\triscaleb}{1.75}
\newcommand*{\triside}{0}
\newcommand*{\intri}{0.7}

\pgfmathsetmacro{\genb}{(\h  / (2*\triscaleb + 1)}
\pgfmathsetmacro{\stri}{\h * \intri}
\pgfmathsetmacro{\gen}{(\stri) / (2*\triscale + 1)}
\pgfmathsetmacro{\s}{\gen}
%\pgfmathsetmacro{\s}{
%\pgfmathsetmacro{\stri}{ (\triscale * \gen) + \s}
\pgfmathsetmacro{\ss}{\genb}

%\begin{scope}[rotate=45]
\foreach \x in {-1,1}
{
  \foreach \y in {-1,1}
  {
    % inner
    \fill({\x * \s},{\y * \s}) circle (4pt);
    \fill({\x * \s},{\y * \stri} ) circle (4pt);
    \fill({\x * \stri},{\y * \s} ) circle (4pt);
    \draw[ultra thick] ({\x*\s},{\y*\s}) -- ({\x*\s},{\y*\stri});
    \draw[ultra thick] ({\x*\s},{\y*\s}) -- ({\x*\stri},{\y*\s});
    \draw[ultra thick] ({\x*\s},{\y*\stri}) -- ({\x*\stri},{\y*\s});
    % outer
    \fill({\x*\ss},{\y*\ss}) circle (4pt);
    \fill({\x*\h},{\y*\ss}) circle (4pt);
    \fill({\x*\ss},{\y*\h}) circle (4pt);
    \draw[ultra thick] ({\x*\ss},{\y*\ss}) -- ({\x*\h},{\y*\ss});
    \draw[ultra thick] ({\x*\ss},{\y*\ss}) -- ({\x*\ss},{\y*\h});
    \draw[ultra thick] ({\x*\h},{\y*\ss}) -- ({\x*\ss},{\y*\h});
    % connect
    \draw[thick] ({\x*\ss},{\y*\ss}) -- ({\x*\s},{\y*\s});
    %edges
    \draw[thick] ({\x*\ss},{\y*\h})--({-1*\x*\ss},{\y*\h});
    \draw[thick] ({\x*\h},{\y*\ss})--({\x*\h},{-1*\y*\ss});
    \draw[thick] ({\x*\s},{\y*\stri})--({-1*\x*\s},{\y*\stri});
    \draw[thick] ({\x*\stri},{\y*\s})--({\x*\stri},{-1*\y*\s});
  }
}
\draw[thick] ({-1*\s},{\stri}) -- ({\s},{-1*\stri});
\draw[thick] ({-1*\stri},{-1*\s}) -- ({\stri},{\s});
\draw[thick] ({\ss},{\h}) -- ({-1*\ss},{-1*\h});
\draw[thick] ({-1*\h},{\ss}) -- ({\h},{-1*\ss});

\draw[thick] ({-1*\s},{\s})--({\ss},{\ss});
\draw[thick] ({\s},{\s})--({\ss},{-1*\ss});
\draw[thick] ({\s},{-1*\s})--({-1*\ss},{-1*\ss});
\draw[thick] ({-1*\s},{-1*\s})--({-1*\ss},{\ss});

\draw[thick] ({-1*\ss},{\h})--({\s},{\stri});
\draw[thick] ({\h},{\ss})--({\stri},{-1*\s});
\draw[thick] ({\ss},{-1*\h})--({-1*\s},{-1*\stri});
\draw[thick] ({-1*\h},{-1*\ss})--({-1*\stri},{\s});
%\end{scope}

%%%%% END MAIN PART %%%%%%

\fill (-7.5,0) circle (4pt);

\fill (-5,0.5) circle (4pt);
\fill (-5,-0.5) circle (4pt);
\fill (-5,1.5) circle (4pt);
\fill (-5,2.5) circle (4pt);
\fill (-5,-1.5) circle (4pt);
\fill (-5,-2.5) circle (4pt);
\draw[thick] (-5,0.5) -- (-5,-0.5);
\draw[thick] (-5,1.5) -- (-5,2.5);
\draw[thick] (-5,-1.5) -- (-5,-2.5);

\draw[thick] (-7.5,0) -- (-5,2.5);
\draw[thick] (-7.5,0) -- (-5,1.5);
\draw[thick] (-7.5,0) -- (-5,0.5);
\draw[thick] (-7.5,0) -- (-5,-0.5);
\draw[thick] (-7.5,0) -- (-5,-1.5);
\draw[thick] (-7.5,0) -- (-5,-2.5);

\draw (4.25,1) -- (3.75,1);
\draw (4.25,0) -- (3.75,0);
\draw (4.25,-1) -- (3.75,-1);

\draw (-4.25,1) -- (-3.75,1);
\draw (-4.25,0) -- (-3.75,0);
\draw (-4.25,-1) -- (-3.75,-1);

\fill (5,0.5) circle (4pt);
\fill (5,-0.5) circle (4pt);
\fill (5,1.5) circle (4pt);
\fill (5,-1.5) circle (4pt);

\node at (-7.5, -3.75) {$(a)$};
\node at (-5, -3.75) {$(b)$};
\node at (0, -3.75) {$(c)$};
\node at (5, -3.75) {$(d)$};

\end{tikzpicture}}
\caption{The four orbits of $\mathrm{Aut}(H_1)$. Parts $(b)$ and $(c)$ are connected
  by $24$ edges, as well as $(c)$ and $(d)$. }
\label{fig:liv}
\end{figure}
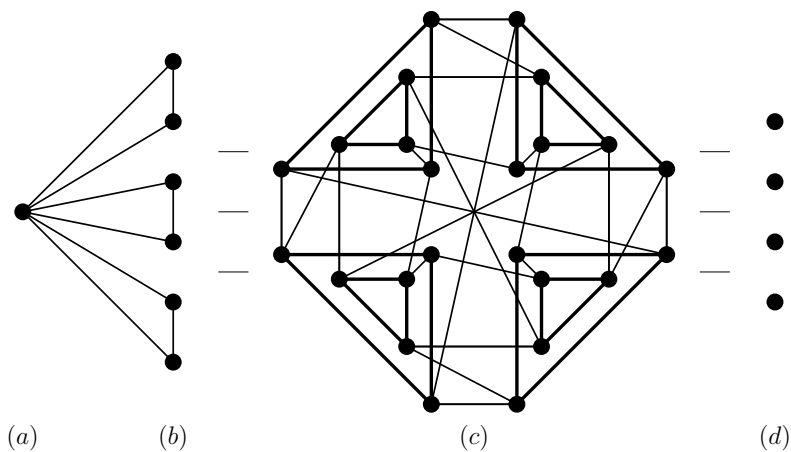

\begin{figure}[htb]
\centering
{\scriptsize 
\setlength{\tabcolsep}{3pt}
\begin{tabular}{r|r|rrrrrr|rrrrrrrrrrrrrrrrrrrrrrrr|rrrr|} 
&\multicolumn{1}{c|}{$a$}&\multicolumn{6}{c|}{$b$}&\multicolumn{24}{c|}{$c$}&\multicolumn{4}{c|}{$d$}\\      \hline
$a$&0&1&1&1&1&1&1&0&0&0&0&0&0&0&0&0&0&0&0&0&0&0&0&0&0&0&0&0&0&0&0&0&0&0&0\\ \hline
\multirow{6}{*}{$b\;$}&1&0&1&0&0&0&0&1&0&0&0&0&0&1&0&0&0&0&0&1&0&0&0&0&0&1&0&0&0&0&0&0&0&0&0\\
&1&1&0&0&0&0&0&0&0&0&1&0&0&0&0&0&1&0&0&0&0&0&1&0&0&0&0&0&1&0&0&0&0&0&0\\
&1&0&0&0&1&0&0&0&0&0&0&1&0&0&0&0&0&1&0&0&0&0&0&1&0&0&0&0&0&1&0&0&0&0&0\\
&1&0&0&1&0&0&0&0&1&0&0&0&0&0&1&0&0&0&0&0&1&0&0&0&0&0&1&0&0&0&0&0&0&0&0\\
&1&0&0&0&0&0&1&0&0&1&0&0&0&0&0&1&0&0&0&0&0&1&0&0&0&0&0&1&0&0&0&0&0&0&0\\
&1&0&0&0&0&1&0&0&0&0&0&0&1&0&0&0&0&0&1&0&0&0&0&0&1&0&0&0&0&0&1&0&0&0&0\\ \hline
\multirow{24}{*}{$c\;$}&0&1&0&0&0&0&0&0&1&1&0&1&0&1&0&0&0&0&0&0&0&0&0&0&0&0&0&0&0&0&0&0&1&0&0\\
&0&0&0&0&1&0&0&1&0&1&0&0&0&0&0&0&0&0&0&0&0&0&0&0&1&0&1&0&0&0&0&1&0&0&0\\
&0&0&0&0&0&1&0&1&1&0&0&0&0&0&0&0&0&0&0&0&0&1&0&0&0&0&0&0&1&0&0&0&0&0&1\\
&0&0&1&0&0&0&0&0&0&0&0&1&1&0&0&1&0&0&0&0&0&0&0&0&0&0&0&0&1&0&0&1&0&0&0\\
&0&0&0&1&0&0&0&1&0&0&1&0&1&0&0&0&0&0&0&0&0&0&0&1&0&0&0&0&0&0&0&0&1&0&0\\
&0&0&0&0&0&0&1&0&0&0&1&1&0&0&0&0&0&0&1&0&0&0&0&0&0&0&1&0&0&0&0&0&0&1&0\\
&0&1&0&0&0&0&0&1&0&0&0&0&0&0&1&1&0&0&0&0&0&0&0&0&0&0&0&0&0&1&0&0&0&1&0\\
&0&0&0&0&1&0&0&0&0&0&0&0&0&1&0&1&0&0&1&0&1&0&0&0&0&0&0&0&0&0&0&0&0&0&1\\
&0&0&0&0&0&1&0&0&0&0&1&0&0&1&1&0&0&0&0&0&0&0&0&0&0&0&0&1&0&0&0&1&0&0&0\\
&0&0&1&0&0&0&0&0&0&0&0&0&0&0&0&0&0&1&1&0&0&0&1&0&0&0&0&1&0&0&0&0&1&0&0\\
&0&0&0&1&0&0&0&0&0&0&0&0&0&0&0&0&1&0&1&1&0&0&0&0&0&0&0&0&0&1&0&1&0&0&0\\
&0&0&0&0&0&0&1&0&0&0&0&0&1&0&1&0&1&1&0&0&0&0&0&0&0&0&0&0&0&0&0&0&0&0&1\\
&0&1&0&0&0&0&0&0&0&0&0&0&0&0&0&0&0&1&0&0&1&1&0&0&0&1&0&0&0&0&0&1&0&0&0\\
&0&0&0&0&1&0&0&0&0&0&0&0&0&0&1&0&0&0&0&1&0&1&0&0&0&0&0&0&0&0&1&0&1&0&0\\
&0&0&0&0&0&1&0&0&0&1&0&0&0&0&0&0&0&0&0&1&1&0&1&0&0&0&0&0&0&0&0&0&0&1&0\\
&0&0&1&0&0&0&0&0&0&0&0&0&0&0&0&0&1&0&0&0&0&1&0&1&1&0&0&0&0&0&0&0&0&1&0\\
&0&0&0&1&0&0&0&0&0&0&0&1&0&0&0&0&0&0&0&0&0&0&1&0&1&1&0&0&0&0&0&0&0&0&1\\
&0&0&0&0&0&0&1&0&1&0&0&0&0&0&0&0&0&0&0&0&0&0&1&1&0&0&0&0&0&0&1&1&0&0&0\\
&0&1&0&0&0&0&0&0&0&0&0&0&0&0&0&0&0&0&0&1&0&0&0&1&0&0&1&1&0&0&0&0&0&0&1\\
&0&0&0&0&1&0&0&0&1&0&0&0&1&0&0&0&0&0&0&0&0&0&0&0&0&1&0&1&0&0&0&0&0&1&0\\
&0&0&0&0&0&1&0&0&0&0&0&0&0&0&0&1&1&0&0&0&0&0&0&0&0&1&1&0&0&0&0&0&1&0&0\\
&0&0&1&0&0&0&0&0&0&1&1&0&0&0&0&0&0&0&0&0&0&0&0&0&0&0&0&0&0&1&1&0&0&0&1\\
&0&0&0&1&0&0&0&0&0&0&0&0&0&1&0&0&0&1&0&0&0&0&0&0&0&0&0&0&1&0&1&0&0&1&0\\
&0&0&0&0&0&0&1&0&0&0&0&0&0&0&0&0&0&0&0&0&1&0&0&0&1&0&0&0&1&1&0&0&1&0&0\\ \hline
\multirow{4}{*}{$d\;$}&0&0&0&0&0&0&0&0&1&0&1&0&0&0&0&1&0&1&0&1&0&0&0&0&1&0&0&0&0&0&0&0&0&0&0\\
&0&0&0&0&0&0&0&1&0&0&0&1&0&0&0&0&1&0&0&0&1&0&0&0&0&0&0&1&0&0&1&0&0&0&0\\
&0&0&0&0&0&0&0&0&0&0&0&0&1&1&0&0&0&0&0&0&0&1&1&0&0&0&1&0&0&1&0&0&0&0&0\\
&0&0&0&0&0&0&0&0&0&1&0&0&0&0&1&0&0&0&1&0&0&0&0&1&0&1&0&0&1&0&0&0&0&0&0 \\ \hline
\end{tabular}}

\caption{Adjacency matrix of $H_1$ separated by the orbits of $\mathrm{Aut}(H_1)$.}
\label{fig:liv_adj}
\end{figure}

\subsection{Higher Parameters}

\begin{theorem} \label{thm:k11}
$39\leq R(C_4,K_{11})\leq 44$.
\end{theorem}

\begin{proof}
The lower bound is obtained by construction. A
$(C_4,K_{11};38)$-graph can easily be obtained by adding a triangle to $H_1$ or $H_2$.

If a $(C_4,K_{11};44)$-graph $G$ exists, then from Lemma \ref{lem:bounds} it
follows that $G$ must have minimum degree at most $7$. Such a graph can be
obtained by applying {\sc Glue} to a $(C_4,K_{10};36)$-graph. However, since
$R(C_4,K_{10}) = 36$, no such graph exists, and therefore $G$ does not exist
as well. 
\end{proof}

\begin{theorem}
$42\leq R(C_4,K_{12})\leq 52$.
\end{theorem}

\begin{proof}
The lower bound is obtained similarly as before, by adding a triangle to the $(C_4,K_{11};38)$-graphs of Theorem
\ref{thm:k11}.

As $R(C_4,K_{11})\leq 44$, any $(C_4,K_{12})$-graph can be obtained by applying
\gl to a $(C_4,K_{11})$-graph with order at most $43$. From Lemma
\ref{lem:bounds}, such a graph must have a minimum degree of at most $7$, and
therefore an order of at most $51$. Thus, $R(C_4,K_{12})\leq 52$.
\end{proof}

\section{Acknowledgments}
This research was done using resources provided by the Open Science Grid, which
is supported by the National Science Foundation and the U.S. Department of
Energy's Office of Science.  We owe many thanks to Mats Rynge for his guidance
throughout our use of the OSG, and to Gurcharan Khanna and Research Computing at
RIT for their valuable and helpful support.

\bibliographystyle{abbrv2}
\bibliography{/home/alex/Dropbox/papers/bibtex/Ramsey,/home/alex/Dropbox/papers/bibtex/Isomorphisms,/home/alex/publications/mendeley/bibtex/Ramsey2,/home/alex/publications/mendeley/bibtex/nauty,/home/alex/publications/mendeley/bibtex/RamseyC4,/home/alex/publications/mendeley/bibtex/Software,/home/alex/Dropbox/papers/bibtex/C4}

\newpage
\appendix
\section*{Appendix: Correctness of Computations}

As our main results relied on the use of algorithms, it was important to take
extra steps to verify the correctness of our implementations. Tests similar to
those described in \cite{McKay1995,Goedgebeur2013} and others were performed, as follows:

\begin{enumerate}
  \item Two independent implementations of {\sc Glue} and {\sc VertexExtend} were
    developed by the authors. The data was generated independently by both
    implementations, and all results agreed. The only data that was not
    produced by both was the full enumeration of
    $\mathcal{R}(C_4,K_8;23)$ and $\mathcal{R}(C_4,K_9;29)$, as this required the most computational
    resources. However, a partial set of $\mathcal{R}(C_4,K_8;23)$, namely when
    $\delta=1,2$, was verified.

  \item Both implementations were used to generate all graphs in
    $\mathcal{R}(C_4,K_t)$ for $4 \leq t \leq 7$. The results agreed, 
    and gave the same counts as those found in \cite{Radziszowski2002}.

  \item \label{test:remove_e} For every $(C_4,K_8;23)$-graph, we removed an edge if it did not
    increase the independence number to $8$, therefore producing a different
    $(C_4,K_8;23)$-graph. Every graph found this way was already included in the
    original set. For example, when going from size $51$ to $50$, $65059062$ of
    the $86582597$ graphs ($\approx 75\%$) were produced, none of which were new.

  \item \label{test:remove_v} For every $(C_4,K_8)$-graph with $24$ and $25$ vertices, every vertex was removed,
    creating a $(C_4,K_8)$-graph with $23$ and $24$ vertices,
    respectively. Every graph produced was already included in the set obtained earlier.

  \item Tests \ref{test:remove_e} and \ref{test:remove_v} were performed on other sets of
    graphs, including $\mathcal{R}(C_4,K_9;29)$. Like before, all graphs
    obtained this way had already been found. 

  \item We extended $\mathcal{R}(C_4,K_9;29)$ to $\mathcal{R}(C_4,K_9;30)$ via
    \ve and also obtained $\mathcal{R}(C_4,K_9;30)=\emptyset$. Similarly, we
    extended the $(C_4,K_{10};35)$-graphs $H_1$ and $H_2$ from Theorem
    \ref{thm:k10} and no $(C_4,K_{10};36)$-graphs were found.

  \item All $(C_4,K_m)$-graphs were independently verified to not contain a $C_4$ or
    independent set of order $m$ using the software {\tt sage} \cite{sage}.
\end{enumerate}

\noindent Most of the large-scale computations were performed on the
Open Science Grid. Over 175000 CPU hours (20 years) were used for these
computations. 

\end{document}